\documentclass[a4paper,12pt]{article}
\usepackage[T1]{fontenc}
\usepackage[utf8]{inputenc} 
\usepackage[english]{babel} 
\usepackage{amssymb, amsthm, amsmath} 
\usepackage{mathtools}						
\mathtoolsset{showonlyrefs,showmanualtags}	
\usepackage{tikz}
\usetikzlibrary{external}
\usepackage{pgfplots}
\usepackage{xcolor}
\usepackage[colorlinks,linkcolor=blue,citecolor=blue,urlcolor=blue]{hyperref}
\usepackage{dsfont}
\usepackage{bm}

\pgfplotsset{compat=1.17}
\usepackage[justification = centering]{caption}
\usepackage{comment}

\newcommand\G{g}
\newcommand\U{\mathcal{U}}
\newcommand\R{\mathbb{R}}
\newcommand\Rn{\mathbb{R}^n}
\newcommand\Rm{\mathbb{R}^k}
\newcommand\Sym{\textnormal{Sym}_+(\Rn)}

\newcommand\E{\mathbb{E}}
\newcommand\Sf{\mathcal{S}}
\newcommand\I{\mathcal{I}}
\DeclareMathAlphabet{\mathpzc}{OT1}{pzc}{m}{it}
\newcommand\f{\mathpzc{f}}
\newcommand\mLin{m}
\newcommand\PLin{P}
\newcommand\mTr{\bm{m}}
\newcommand\PTr{\bm{P}}

\newcommand\m{\mu}
\newcommand\hx{m}

\newcommand\Vect[1]{\mathrm{span}\left\lbrace #1\right\rbrace}
\newcommand\vecdef[2]{\begin{pmatrix} #1 \\ #2 \end{pmatrix}}

\theoremstyle{plain}
\newtheorem{theo}{Theorem}
\newtheorem{prop}[theo]{Proposition}
\newtheorem{lem}[theo]{Lemma}
\newtheorem{coro}[theo]{Corollary}

\newtheorem{assu}{Assumption}
\theoremstyle{definition}
\newtheorem{prob}{Problem}
\newtheorem{rem}{Remark}

\usepackage{todonotes}  

\reversemarginpar

\title{Statistical Linearization for Robust Motion Planning }

\author{C. Leparoux\thanks{UMA, ENSTA Paris, Institut Polytechnique de Paris, 91120 Palaiseau, France; and DTIS, ONERA, Université Paris-Saclay, 91123 Palaiseau, France. \textit{E-mail}: {\tt clara.leparoux@ensta-paris.fr}.} \and R. Bonalli\thanks{Laboratoire des Signaux et Syst\`emes, Université Paris-Saclay, CNRS, CentraleSupélec, Bât. Bréguet, 3 Rue Joliot Curie, 91190 Gif-sur-Yvette, France. \textit{E-mail}: {\tt riccardo.bonalli@cnrs.fr}.} \and  B. H\'eriss\'e\thanks{DTIS, ONERA, Université Paris-Saclay, F-91123 Palaiseau, France. \textit{E-mail}: {\tt bruno.herisse@onera.fr}.} \and F. Jean\thanks{UMA, ENSTA Paris,
        Institut Polytechnique de Paris, 91120 Palaiseau, France. \textit{E-mail}: {\tt frederic.jean@ensta-paris.fr}.}}
\begin{document}
\maketitle

\begin{abstract}
The goal of robust motion planning consists of designing open-loop controls which optimally steer a system to a specific target region while mitigating uncertainties and disturbances which affect the dynamics. Recently, stochastic optimal control has enabled particularly accurate formulations of the problem. Nevertheless, despite interesting progresses, these problem formulations still require expensive numerical computations. In this paper, we start bridging this gap by leveraging statistical linearization. Specifically, through statistical linearization we reformulate the robust motion planning problem as a simpler deterministic optimal control problem subject to additional constraints. We rigorously justify our method by providing estimates of the approximation error, as well as some controllability results for the new constrained deterministic formulation. Finally, we apply our method to the powered descent of a space vehicle, showcasing the consistency and efficiency of our approach through numerical experiments.
\end{abstract}

\section{Introduction}

Motion planning is a powerful tool for motion design, with important applications in engineering and biology. Its main objective consists of computing an open-loop control which steers a given system to some desired target, while possibly optimizing performance criteria, e.g., minimizing effort. Specifically, motion planning becomes essential when no feedback-based control strategies are available, for instance because either no measurements are available or some states are not observable, as it happens in particular in the context of fast biological movements \cite{Berret20201}. In addition, motion planning is particularly beneficial to compute reference strategies low-level controllers track at online later stages \cite{Malyuta2021a}, an approach which is of common use in robotics and aerospace \cite{ganet2019, Scharf2017}. 
In all the aforementioned tasks, uncertainties ranging from measurement errors to unknown parameters or external perturbations may hinder the reliability of the computed control strategies. Methods taking these uncertainties into account are referred to as robust motion planning.

The existing robust motion planning methods can be categorized into three groups. In the first group, the so-called robust set methods tackle uncertainty by representing the states of the system through sets which contain all the possible outcomes. In particular, these paradigms build on interval analysis \cite{Vehi2002}, and they have been successfully applied in both robotics \cite{Pepy2009} and aerospace \cite{Cheng2019, Bertin2021}. Similarly, we find works which leverage positively invariant sets to generate safe trajectories \cite{Blanchini1999, Berntorp2017, Danielson2020}. The common drawback of this class of methods is that they generally produce either conservative solutions or computationally expensive algorithms \cite{Blanchini1999}. The second group of works includes methods that reduce the sensitivity with respect to uncertain parameters of the computed control strategies for specific metrics. For instance, \cite{Rustem1994, Darlington2000} devise algorithms which minimize the sensitivity of the performance criteria which are to optimize. Alternatively, the algorithm in \cite{Plooij2015} minimizes the sensitivity of the final state when controlling a robotic arm. Finally, \cite{Seywald2019} models the dynamics of the sensitivity and minimizes the latter via an optimal control approach. Note that methods designed to treat uncertain parameters are surveyed in \cite{Nagy2004}.
These first two classes of methods are well suited to handle parameter uncertainties, but random perturbations of the dynamics are better addressed by the third and last group of works which leverage the so called stochastic methods. More specifically, stochastic methods ensure robustness with respect to uncertain outcomes by reducing the corresponding covariance. Notably, \cite{Sain1966} has first provided the theoretical solution to the problem of minimizing the variance of the cost of an LQ problem under disturbances on the inputs. On the other hand, \cite{Hotz1985} and \cite{Chen2015} propose analysis for exact covariance steering. Finally, \cite{Censi2008} achieves covariance minimization through penalization of the covariance of the final state within the cost, with the help of information matrices. Nevertheless, despite their accuracy, stochastic methods generally suffer from expensive numerical computations (except maybe for the linear quadratic case). Moreover, most of the existing methods have been tailored to specific applications, therefore we may claim there is no systematic methodology for robust motion planning using the stochastic modeling.

In this paper, we propose a paradigm for robust motion planning which methodologically belongs to the aforementioned third group of works, and which in particular leverages stochastic differential equations to model uncertainty essentially along two main steps. First, as presented in \cite{Berret20202}, we model the motion planning problem through a stochastic open-loop optimal control problem where the state covariance is penalized in the cost to ensure robustness. Then, we approximate this latter stochastic formulation through an appropriate deterministic optimal control problem whose state variables are the first two moments of the original stochastic state, i.e., its mean and covariance. These mean and covariance are efficiently computed thanks to statistical linearization, which essentially boils down to approximating the distribution of the original stochastic state through a Gaussian distribution. Statistical linearization methods have been successfully used for decades \cite{socha2007, Elishakoff2016}, especially for applications in mechanical \cite{roberts2003}. In addition, \cite{Lambert2022} recently showed statistical linearization may be beneficially leveraged for variational inference. Despite its sound numerical performance, the theoretical well-posedness of statistical linearization still remains an open question. Specifically, although accessibility properties of statistical linearization have already been preliminary studied \cite{bonalli2022}, thus motivating its use for approximate covariance steering, the fidelity of approximations stemming from statistical linearization still requires investigation. In this paper, we start bridging this gap by computing estimates of the approximation error which is generated by statistical linearization. We additionally study the controllability of statistical linearization when accuracy constraints are imposed.

To showcase the validity of our theoretical findings, we solve a robust motion planning problem for the powered descent of a space vehicle. In the recent literature, the resulting optimal control problem has almost exclusively been studied in deterministic settings, i.e., uncertainties have not so far been considered. For this deterministic setting, theoretical analysis shows that the optimal control has generally a Max-Min-Max form \cite{Leparoux2022c, Lu2018, Gazzola2021}, a control law which has been shown to suffer from lack of robustness with respect to uncertainties and disturbances \cite{Ridderhof2019}. Therefore, infusing robustness with respect to uncertainty in the aforementioned optimal control problem formulation is crucial to ensure reliability in applications which are subject to high performance criteria. To the best of our knowledge, one of the first work addressing uncertainty for the powered descent of a space vehicle is \cite{Shen2010}, in which the authors make use of multi-objective optimization to minimize fuel consumption and state sensitivities at the same time. This strategy unfortunately increases the number of state variables by the square of the dimension of the state, therefore hindering computational performance. To reduce computational burden, in the more recent work \cite{Ridderhof2019} the first two moments of the state variables are controlled via two separated deterministic problems: a mean steering problem and a covariance steering problem. Nevertheless, the aforementioned robust motion planning methods lack justification of well-posedness, as well as analysis of the accuracy of the corresponding approximations. We fill this gap by applying our robust motion planning method for the powered descent of a space vehicle which explicitly includes modeling of environmental and system uncertainties, such as aerodynamic effects, parameter uncertainties, and measurement errors. We study the accessibility of the statistically linearized dynamics under both open-loop and partial feedback control laws, and we propose modelization for actuator limits in a stochastic setting. Finally, numerical results are provided to illustrate the consistency of the approach.

This paper is organized as follows. In Section 2, we formulate the robust motion planning problem and propose an approach which leverage statistical linearization to simplify it. Then, in Section 3 we compute estimates for the error induced by statistical linearization and we study the controllability of the approximated robust motion planning problem under new feasibility constraints which make statistical linearization well-posed. Finally, in Section 4 we apply our method to the powered descent of a space vehicle, providing numerical results which sustain our theoretical findings.

\section{Modelling robustness via stochastic open-loop optimal control} \label{se:robustMP}

As we mentioned in the introduction, we start by modeling a motion planning problem via a deterministic control system 
\begin{equation}\label{eq:dyn} 
\dot x = f(x,u), 
\end{equation}
where $ x \in \R^n$ is the state variable and $ u \in \mathcal{U} \subset \R^k$ is the control variable. We aim finding an open-loop control law $u(t)$, $t \in [0,t_f]$, which steers the system from an initial state $x^0$ to a target set $\Sf \subset \R^n$, while minimizing a certain cost. The motion planning problem then amounts to the following optimal control problem (note that the final time $t_f$ below may be free or fixed depending on the considered problem).

\begin{prob}[Deterministic motion planning] \label{prob1}
\begin{equation}
\min  C(u)=\psi(x(t_f)) + \int_0^{t_f} L(x(t),u(t)) dt
\end{equation}
among the controls $u \in L^2([0,t_f],\mathcal{U})$ such that the solution $x$ of
\begin{equation}
			\dot x(t)=f(x(t),u(t))  \ \hbox{a.e.\ on } [0,t_f],  \qquad
			x(0) =  x^0, 			
\end{equation}
satisfies $x(t_f) \in \Sf$.
\end{prob}

At this step, let us further assume that the system is affected by uncertain perturbations. We consider two kind of perturbations here: 1) uncertainties on the initial state and 2) noise on the dynamics. We model the first uncertainties by taking a random variable $x^0$ as initial state; the second ones are modelled by replacing the initial control system~\eqref{eq:dyn} with an Itô-type stochastic control system
\begin{equation} \label{eq:sde}
    dx_t = f(x_t,u(t)) dt + \G(x_t,u(t)) dW_t,
\end{equation} 
where $W_t$ is a $d$-dimensional Wiener process and $\G(x,u)$ is a dispersion matrix.
The quantities $x(t_f)$ and $C(u)$ are now random variables, and hence the corresponding terminal conditions and cost should be expressed via appropriate expectations.

In this setting the sample trajectories are dispersed around the mean trajectory, and the dispersion may increase along these trajectories. Consequently the final states of sample trajectories may lie far from the target set $\mathcal{S}$. It is then natural to seek robustness by penalizing the norm of the covariance $\PTr(t)$ of the state $x_t$ in the cost. 
For this, let us introduce two non-negative symmetric matrices $\bar{Q}_f$ and $\bar{Q}$ which quantify penalization of the final covariance and of the covariance along trajectories, respectively. We thus obtain a first way of constructing a robust open-loop control strategy as solution to the following stochastic open-loop optimal control problem.

\begin{prob}[Robust motion planning by stochastic optimal control]
\label{pb:sto}
\begin{equation}
\min J(u)= \E[C(u)] + \mathrm{tr} (\bar{Q}_f\PTr(t_f)) + \int_0^{t_f} \mathrm{tr} (\bar{Q} \PTr(t)) dt
\end{equation}
among the controls $u \in L^2([0,t_f],\mathcal{U})$ such that the solution $x_t$ of
\begin{equation}
	\begin{array}{l}
			dx_t = f(x_t,u(t)) dt + \G(x_t,u(t)) dW_t,  \quad t \in [0,t_f], \\[2mm]
			x_0 =  x^0 \ \hbox{(random variable)},		
	\end{array}										
\end{equation}
satisfies $\E[x_{t_f}]  \in \Sf$.
\end{prob}

Solving such a stochastic optimal control problem is computationally difficult and few theoretical and numerical methods exist. However since mean and covariance play a major role in the above problem, we can adopt the approach presented in~\cite{Berret20201}, which is based on statistical linearization.

The main idea of this approach is to approximate the mean $\mTr(t)$ and covariance $\PTr(t)$ of the process $x_t$  by the solution $(\mLin,\PLin)$ of the following deterministic control system, which is called \emph{statistical linearization} of~\eqref{eq:sde}:
\begin{equation} 
\begin{cases}
\dot{\hx} &= f(\hx,u)\\
\dot{\PLin} &= D_xf(\hx,u) \PLin + \PLin D_xf(\hx,u)^{\top} + \G(\hx,u) \G(\hx,u)^{\top}.
\end{cases}
\end{equation}
Importantly, the method requires that the deterministic cost $C$ in Problem \ref{prob1} has a quadratic dependence with respect to the state variables, so from now on we assume that the following property holds.
\begin{itemize}
    \item[(H)] The terminal cost $\psi$ and the infinitesimal cost $L$ in Problem \ref{prob1} are quadratic functions (non necessarily homogeneous) of $x$. 
\end{itemize}
Assumption (H) guarantees that the expectation of the cost $C$ in Problem \ref{prob1} depends on $\mTr$ and $\PTr$ uniquely. Indeed, by denoting with $Q_{\psi}$ and $Q_L(u)$ the symmetric matrices which represent the homogeneous quadratic parts of $\psi$ and $L$ respectively, there holds
\begin{multline}
  \E[C(u)] =  \psi(\mTr(t_f)) + \mathrm{tr} (Q_{\psi} \PTr(t_f)) \\ + \int_0^{t_f} \left( L(\mTr(t),u(t)) + \mathrm{tr} (Q_L (u(t))\PTr(t)) \right) dt.
\end{multline}
Setting $Q_f=\bar{Q}_f + Q_{\psi}$ and $Q(u)=\bar{Q} + Q_L(u)$, we finally obtain that
\begin{equation}
  J(u) =  \psi(\mTr(t_f)) + \mathrm{tr} (Q_f \PTr(t_f)) + \int_0^{t_f} \left( L(\mTr(t),u(t)) + \mathrm{tr} (Q (u(t))\PTr(t)) \right) dt.
\end{equation}
Therefore, we propose to compute an open-loop control for the robust motion planning Problem \ref{pb:sto} by rather solving the following deterministic optimal control problem (below, $\mTr^0$ and $\PTr^0$ denote the mean and covariance of the random variable $x^0$, respectively).

\begin{prob}[Robust motion planning by statistical linearization]
\label{pb:stat_lin}
\begin{equation}
\min  J_{\mathit{lin}}(u)= \psi(\mLin(t_f)) + \mathrm{tr} (Q_f \PLin(t_f)) + \int_0^{t_f} \left( L(\mLin(t),u(t)) + \mathrm{tr} (Q (u(t))\PLin(t)) \right) dt
\end{equation}
among the controls $u \in L^2([0,t_f],\mathcal{U})$ such that the solution $(\mLin,\PLin)$ of 
\begin{equation} \label{eq:syslin}
\begin{array}{l}
\left\{ \begin{array}{ll}
\dot{\hx} &= f(\hx,u),\\[2mm]
\dot{\PLin} &= D_xf(\hx,u) \PLin + \PLin D_xf(\hx,u)^{\top} + \G(\hx,u) \G(\hx,u)^{\top},
\end{array} \right. \\[6mm]
\mLin(0) = \mTr^0, \quad \PLin(0)=\PTr^0,
\end{array}
\end{equation}
satisfies $\mLin(t_f)  \in \Sf$.
\end{prob}

\begin{rem} 
The formulation in Problem \ref{pb:stat_lin} has two additional notable features: 
\begin{itemize}
    \item the mean variable (equation in $\mLin$) and the covariance variable (equation in $\PLin$), which measures the dispersion of the process $x_t$, are decoupled; 
    \item the resulting trajectory of the mean $\mLin$ corresponds to the planned trajectory for the deterministic system.
\end{itemize} 
\end{rem}

Despite its simplicity, the formulation in Problem \ref{pb:stat_lin} raises several questions, both from theoretical and practical point of view. First, in which sense is the solution to Problem \ref{pb:stat_lin} an approximation of the solution to Problem \ref{pb:sto}?
We answer this question in Section~\ref{se:justification}.

Second, is this approach relevant for robustness, in particular is the use of statistical linearization sufficient to reduce the covariance? The latter question  boils down to study the controllability properties of the statistical linearization, i.e., the subject of our paper \cite{bonalli2022}. From a practical point of view, we show the efficiency which is offered by solving Problem \ref{pb:stat_lin} by studying the example of the landing of a reusable launcher in Section~\ref{se:launcher}.

\section{Justification of statistical linearization}
\label{se:justification}

In this section, our objective consists of endowing statistical linearization with guarantees of well-posedness, and for this we aim at showing two important properties under mild assumptions. First, we compute estimates which quantify constraints for the approximation error between the trajectories of the stochastic control system \eqref{eq:sde} and those of its statistical linearization \eqref{eq:syslin}. In short, this property ensures statistical linearization represents a well-posed approximation of mean and covariance of a stochastic control system as soon as the variance is forced to take small values. As a second result, we show that statistical linearization \eqref{eq:syslin} is controllable when the aforementioned approximation constraints are considered, in the specific case of control-linear systems, a fairly general class of control systems widely used in application. This latter result ensures that, under additional approximation constraints, Problem \ref{pb:stat_lin} is feasible, justifying the search of an optimal solution, and their solution trajectories are close to the trajectories of the original stochastic control system.

\subsection{Estimates for statistical linearization}

In this section, we leverage the following arguably mild assumption.

\begin{assu} \label{assu:estimate}
The following hold for \eqref{eq:sde}:
\begin{itemize}
    \item the differential of the dynamics $f$ is bounded uniformly-in-state, i.e., there exists a continuous function $\varphi : \R_+\to\R_+$ such that $\| D_xf(x,u) \| \le \varphi(\| u \|)$, for every $(x,u) \in \Rn\times\U$;

    \item the dispersion matrix does not depend on the state variable, i.e., it holds that $\G(x,u) = \G(u)$, for every $(x,u) \in \Rn\times\U$.
\end{itemize}
\end{assu}

Note that, on the one hand the first condition in Assumption \ref{assu:estimate} is standard, in that it is generally required, among other conditions, to ensure existence and uniqueness for solutions to \eqref{eq:sde}. Also, this condition trivially holds for vector field $f$ with compact support (in which case the function $\varphi$ is just a constant), a property often implicitly assumed in the models. On the other hand, the second condition in Assumption \ref{assu:estimate} is for instance satisfied as soon as uncertainty (specifically, $\G(x_t,u(t)) dW_t$ in \eqref{eq:sde}) originates from some uniformly recurrent noise, such as system measurement errors, uncertain weather conditions, etc.

Before establishing estimates of the error between the trajectories of the stochastic control system and those of its statistical linearization, we note that under Assumption \ref{assu:estimate} the original stochastic control system takes the form
\begin{equation} \label{eq:trueSDE}
    \mathrm{d}x_t = f(x_t,u(t)) \, \mathrm{d}t + \G(u(t)) \, \mathrm{d}W_t ,
\end{equation}
and the statistical linearization of \eqref{eq:trueSDE} now writes
\begin{equation} \label{eq:statLinSDE}
    \begin{cases}
    \dot{\mLin}(t) = f(\mLin(t),u(t)) \\
    \dot{\PLin}(t) = D_x f(\mLin(t),u(t)) \PLin(t) + \PLin(t) D_x f(\mLin(t),u(t))^{\top} + \G(u(t)) \G(u(t))^{\top} .
    \end{cases}
\end{equation}


Error estimates between the trajectories of the stochastic control system \eqref{eq:trueSDE} and those of its statistical linearization \eqref{eq:statLinSDE} are given below. Let us first fix some notation. Given a control law $u \in L^2([0,t_f],\U)$ and a random variable $x^0$ with mean and covariance $(\mTr^0,\PTr^0)$, we denote by $(\mTr,\PTr)$ the mean and covariance of the solutions to \eqref{eq:trueSDE} with initial condition $x^0$ and control $u$, whereas $(\mLin,\PLin)$ denotes the trajectories solutions to \eqref{eq:statLinSDE} with initial condition $(\mTr^0,\PTr^0)$ and control $u$.

\begin{prop} \label{prop:estimate}
Assume Assumption \ref{assu:estimate} holds and fix $t_f>0$. Then there exists an increasing function $\alpha: \R^+ \to \R^+$ such that for every $u$ and $x^0$,  the following estimate holds true,
\begin{multline} \label{eq:trueBound}
     \underset{t \in [0,t_f]}{\sup} \ \| \mTr(t) - \mLin(t) \|^{2} + \underset{t \in [0,t_f]}{\sup} \ \| \PTr(t) - \PLin(t) \| \le  \\
     \alpha \left( \int^{t_f}_0 \varphi(\| u(s) \|)  \mathrm{d}s \right) \times   \int^{t_f}_0 \varphi(\| u(s) \|) \| \PLin(s) \| \; \mathrm{d}s,
\end{multline}
provided all quantities are defined in the whole interval $[0,t_f]$.
\end{prop}

\begin{rem}
The error on the mean has an exponent 2 for homogeneity reasons, the covariance being homogeneous to the square of the mean.
\end{rem}

\begin{proof}
Since the curves $\mLin$ and $\mTr$ originate from the same initial condition $\mTr^0$, the difference $\mTr(t)-\mLin(t)$ writes as
\begin{equation}
\int_0^t \left(\dot{\mTr}(s) - \dot{\mLin}(s) \right)ds = \int_0^t \left( \mathbb{E}[ f(x_s,u(s)) ] - f(\mLin(s),u(s)) \right)ds.
\end{equation}
 Thanks to Assumption \ref{assu:estimate} and a first-order Taylor development of $f$ with integral rest, for every $t \in [0,t_f]$ we may compute (we implicitly overload the constant $C > 0$ in what follows)
\begin{align*}
    &\| \mTr(t) - \mLin(t) \| \le \int^t_0 \| \mathbb{E}[ f(x_s,u(s)) ] - f(\mLin(s),u(s)) \| \; \mathrm{d}s \\
    &\le \int^t_0 \| f(\mTr(s),u(s)) - f(\mLin(s),u(s)) \| \; \mathrm{d}s \\
    &\quad + \int^t_0 \int^1_0 \left\| \mathbb{E}\Big[ D_x f\big(\theta x_s + (1 - \theta) \mTr(s) , u(s)\big) (x_s - \mTr(s)) \Big] \right\| \; \mathrm{d}\theta \; \mathrm{d}s \\
    &\le C \int^t_0 \varphi(\| u(s) \|) \big( \| \mTr(s) - \mLin(s) \| + \mathbb{E}[ \| x_s - \mTr(s) \| ] \big) \; \mathrm{d}s \\
    &\le C \int^t_0 \varphi(\| u(s) \|) \left( \| \mTr(s) - \mLin(s) \| + \big( \textnormal{tr} \PTr(s) \big)^{\frac{1}{2}} \right) \; \mathrm{d}s ,
\end{align*}
where in the last line we use H\"older's inequality and the fact that, for every $s \in [0,t_f]$, by definition $\mathbb{E}[ \| x_s - \mTr(s) \|^2 ] = \textnormal{tr} \PTr(s)$. Therefore, a routine application of Gr\"onwall's inequality yields
\begin{equation} \label{eq:mBound}
    \underset{t \in [0,t_f]}{\sup} \ \| \mTr(t) - \mLin(t) \| \le \alpha \left( \int^{t_f}_0 \varphi(\| u(s) \|)  \mathrm{d}s \right)   \int^{t_f}_0 \varphi(\| u(s) \|) \big( \textnormal{tr} \PTr(t) \big)^{\frac{1}{2}} \; \mathrm{d}t ,
\end{equation}
where $\alpha$ can be taken as the function $\alpha(s)=C e^{Cs}$ here.

A similar computation may be performed for the covariance. Recall that, as a result of a classical computation using It\^{o}'s formula, there holds
\begin{equation}
\dot{\PTr}(t) = \mathbb{E}\big[ f(x_t,u(t)) (x_t - \mTr(t))^{\top} + (x_t - \mTr(t)) f(x_t,u(t))^{\top} \big] + \G(u(t)) \G(u(t))^{\top} .
\end{equation}
As a consequence, for every $t \in [0,t_f]$ we may compute
\begin{align*}
    &\| \PTr(t) - \PLin(t) \| \le \int^t_0 \left\| \mathbb{E}[f(x_s,u(s)) (x_s - \mTr(s))^{\top}] - D_x f(\mLin(s),u(s)) \PLin(s) \right\| \; \mathrm{d}s \\
    &\hspace{13ex}+ \int^t_0 \left\| \mathbb{E}[(x_s - \mTr(s)) f(x_s,u(s))^{\top}] - \PLin(s) D_x f(\mLin(s),u(s))^{\top} \right\| \; \mathrm{d}s \\
    &= 2 \int^t_0 \bigg\| \int^1_0 \mathbb{E}\Big[ D_x f\big( \theta x_s + (1 - \theta) \mTr(s) , u(s) \big) ( x_s - \mTr(s) ) ( x_s - \mTr(s) )^{\top} \\
    &\quad- D_x f(\mLin(s),u(s)) \PLin(s) \pm D_x f(\mLin(s),u(s)) ( x_s - \mTr(s) ) ( x_s - \mTr(s) )^{\top} \Big] \; \mathrm{d}\theta \bigg\| \; \mathrm{d}s \\
    &\le C \int^t_0 \varphi(\| u(s) \|) \big( \mathbb{E}\left[ \| x_s - \mTr(s) \|^2 \right] + \| \PTr(s) - \PLin(s) \| \big) \; \mathrm{d}s ,
\end{align*}
where in the first equality the symbol $\pm$ means that we added and subtracted the quantity
$$
D_x f(\mLin(s),u(s)) ( x_s - \mTr(s) ) ( x_s - \mTr(s) )^{\top} , \quad s \in [0,t_f] .
$$
Noting that
\begin{equation} \label{eq:trace_norme}
\textnormal{tr} \PTr(s) \le \| \PTr(s) \| \le \| \PTr(s) - \PLin(s) \| + \| \PLin(s) \| , \quad s \in [0,t_f] ,
\end{equation}
a routine application of Gr\"onwall's inequality readily yields  (we implicitly overload the nondecreasing function $\alpha$ in what follows)
\begin{equation} \label{eq:PBound}
    \underset{t \in [0,t_f]}{\sup} \ \| \PTr(t) - \PLin(t) \| \le \alpha \left( \int^{t_f}_0 \varphi(\| u(s) \|)  \mathrm{d}s \right)  \int^{t_f}_0 \varphi(\| u(s) \|) \| \PLin(s) \| \; \mathrm{d}s .
\end{equation}
At this step, we note that from \eqref{eq:mBound} we in particular infer that (we write $\alpha$ instead of $\alpha \left( \int^{t_f}_0 \varphi(\| u(s) \|)  \mathrm{d}s \right)$ and overload this for readability)
\begin{align*}
    &\underset{t \in [0,t_f]}{\sup} \ \| \mTr(t) - \mLin(t) \|^2 \le \alpha \times\left( \int^{t_f}_0 \varphi(\| u(s) \|) \big( \textnormal{tr} \PTr(s) \big)^{\frac{1}{2}} \; \mathrm{d}s \right)^2\\
    &\le \alpha \times \left(\int^{t_f}_0 \varphi(\| u(s) \|) \ \left( \| \PTr(s) - \PLin(s) \|^{\frac{1}{2}} + \| \PLin(s) \|^{\frac{1}{2}} \right) \; \mathrm{d}s \right)^2\qquad \hbox{by \eqref{eq:trace_norme},} \\
    &\le \alpha \times \left( \left( \int^{t_f}_0 \varphi(\| u(s) \|) \| \PLin(s) \| \; \mathrm{d}s \right)^{\frac{1}{2}}+ \int^{t_f}_0 \varphi(\| u(s) \|) \| \PLin(s) \|^{\frac{1}{2}} \; \mathrm{d}s \right)^2 , \\
    &\le \alpha \times \left(  \int^{t_f}_0 \varphi(\| u(s) \|) \| \PLin(s) \| \; \mathrm{d}s + \left(\int^{t_f}_0 \varphi(\| u(s) \|) \| \PLin(s) \|^{\frac{1}{2}} \; \mathrm{d}s \right)^2 \right) .
\end{align*}
By the Cauchy-Schwartz inequality, we get
\begin{equation}
  \left(\int^{t_f}_0 \varphi(\| u(s) \|) \| \PLin(s) \|^{\frac{1}{2}} \mathrm{d}s \right)^2 \leq  \int^{t_f}_0 \varphi(\| u(s) \|)  \mathrm{d}s  \int^{t_f}_0 \varphi(\| u(s) \|) \| \PLin(s) \| \mathrm{d}s  ,
\end{equation}
and we finally obtain
\begin{equation}
\underset{t \in [0,t_f]}{\sup} \ \| \mTr(t) - \mLin(t) \|^2 \le \alpha \times     \int^{t_f}_0 \varphi(\| u(s) \|) \| \PLin(s) \| \; \mathrm{d}s ,
\end{equation}
which combined with \eqref{eq:PBound} yields \eqref{eq:trueBound}, and the conclusion follows.
\end{proof}

Estimate \eqref{eq:trueBound} becomes particularly insightful when controls are bounded. In this setting, the subset $\varphi( \| \U \| ) \subseteq \R_+$ being bounded, Proposition \ref{prop:estimate} essentially states that the unique leverage to minimize the error generated by statistical linearization consists of forcing the $L^1$ norm of $\PLin$ to stay small. The latter property happens to be particularly useful in many applications which range from aerospace and robotics \cite{Ridderhof2019, Seywald2019}, to estimation and learning \cite{Lambert2022}.

\begin{theo} \label{theo:estimate}
Assume Assumption \ref{assu:estimate} holds and the control set $\U$ is bounded. Fix $t_f>0$. 
Then there exists a constant $C > 0$ which depends on $f$, $t_f$, and $\U$ uniquely and for which the following estimate holds true,
\begin{equation} \label{eq:trueBoundFinal}
\underset{t \in [0,t_f]}{\sup} \ \| \mTr(t) - \mLin(t) \|^2 + \underset{t \in [0,t_f]}{\sup} \ \| \PTr(t) - \PLin(t) \|
    \le C  \int^t_0 \| \PLin(s) \| \; \mathrm{d}s ,
\end{equation}
provided all quantities are defined in the whole interval $[0,t_f]$.
\end{theo}


\subsection{Controllability of Problem \ref{pb:stat_lin} under approximation constraints} \label{sec:contr}

Proposition \ref{prop:estimate} quantitatively estimates the well-posedness of statistical linearization: if one is capable of controlling \eqref{eq:statLinSDE} in such a way that the left-hand side of~\eqref{eq:trueBound} takes small values, then the solutions to the statistical linearization well-approximate, i.e., with respect to second-order moments, the solutions to the original stochastic control system. These latter requirements represent new constraints to which Problem \ref{pb:stat_lin} should be subject, to compute meaningful solutions to Problem \ref{pb:sto}.

However the aforementioned constraints entail controllability issues. Indeed Problem \ref{pb:stat_lin} consists of a minimization problem on the set of $L^2$ controls whose associated trajectory satisfies the terminal condition $\mLin(t_f) \in \Sf$. The fact that this set of controls is non-empty can be guaranteed by usual controllability conditions on the dynamics $f(x,u)$. The question is whether this set of admissible control remains non-empty when adding the new constraints we previously introduced. The aim of this section is to answer this question.

Let us first formalize the setting properly. For this, fix $t_f >0$, and choose a vector $\mLin^0 \in \Rn$ and a non-nnegative symmetric matrix $\PLin^0$. Given a control $u \in L^2([0,t_f],\Rm)$, we denote by $(\mLin_u,\PLin_u)$ the solution to \eqref{eq:syslin} with initial condition $(\mLin^0,\PLin^0)$ and control $u$. Assume Assumption \ref{assu:estimate} is satisfied and let $\alpha$ be the increasing function  stemming from Proposition~\ref{prop:estimate}. For $\varepsilon >0$ we define $\mathbb{U}(\varepsilon, \mLin^0, \PLin^0)$ to be the set of controls $u \in L^2([0,t_f],\Rm)$ such that
\begin{equation}
\label{eq:approx_constraint}
\alpha \left( \int^{t_f}_0 \varphi(\| u(s) \|)  \mathrm{d}s \right) \times  \int^{t_f}_0 \varphi(\| u(s) \|) \| \PLin(s) \| \; \mathrm{d}s  \leq \varepsilon.
\end{equation}

As a direct corollary of Proposition~\ref{prop:estimate}, $\mathbb{U}(\varepsilon, \mLin^0, \PLin^0)$ is a class of control guaranteeing that the statistical linearization approximates to within $\varepsilon$ the solutions to the original stochastic control system.

\begin{coro}
Assume Assumption \ref{assu:estimate} holds and consider a random variable $x^0$ with mean and covariance equal to  $(\mLin^0,\PLin^0)$ respectively. Given a control $u \in L^2([0,t_f],\Rm)$,  we denote by $(\mTr_u,\PTr_u)$ the mean and the covariance of the solution to \eqref{eq:sde} with initial condition $x^0$ and control $u$, respectively. Then, for any $u \in \mathbb{U}(\varepsilon, \mLin^0, \PLin^0)$ there holds
\begin{equation}
\underset{t \in [0,t_f]}{\sup} \ \| \mTr_{u}(t) - \mLin_{u}(t) \|^2 + \underset{t \in [0,t_f]}{\sup} \ \| \PTr_{u}(t) - \PLin_{u}(t) \|  \leq \varepsilon .
\end{equation}
\end{coro}

Our initial question is whether for any $\varepsilon >0$ one can find controls $u$ in $\mathbb{U}(\varepsilon, \mLin^0, \PLin^0)$ satisfying  $\mLin_u(t_f) \in \Sf$. We consider the most constrained case where the target is reduced to a point, i.e., $\Sf = \{ \mLin^f\}$. Moreover we restrict ourselves to a more specific setting than the one of Assumption~\ref{assu:estimate}.

\begin{assu} \label{assu:controlAffine}
The following hold for \eqref{eq:sde}:
\begin{enumerate}
    \item the control domain is $\U = \Rm$;

    \item the dynamics $f$ is control-linear, i.e.,
    $$
    f(x,u) = \sum^k_{i=1} u_i f_i(x) , \quad  (x,u) \in \Rn\times\Rm ,
    $$
    where  $f_i : \Rn \to \Rn$, $i=1,\dots,k$, are smooth vector fields for which there exists $L > 0$ such that $\| D_x f_i(x) \| \le L$ for every $x \in \Rn$, $i=1,\dots,k$;

    \item the vector fields satisfy the Lie bracket generating condition
    \begin{equation} \label{eq:Lie}
        \Rn = \textnormal{Lie}(f_1,\dots,f_k)(x) , \quad \textnormal{for every} \ x \in \Rn ;
    \end{equation}

    \item the dispersion is constant, i.e., $\G(x,u) = \G$, for every $(x,u) \in \Rn\times\Rm$.
\end{enumerate}
\end{assu}

Note that the above assumption implies that Assumption~\ref{assu:estimate} is satisfied with $\varphi(r) = L r$. Assumption \ref{assu:controlAffine} may seem rather restrictive, in particular for what concerns the hypothesis that the dynamics is control-linear and the dispersion is constant. However this setting is encountered in many control applications, often in the case of simplification of more complex problems. Also, the Lie bracket generating condition \eqref{eq:Lie} is \textit{generic}, i.e., it is satisfied by almost every tuple of vector fields $(f_1,\dots,f_k)$.

Under this assumption we show that we have controllability for the variable $\mLin$ in any class $\mathbb{U}(\varepsilon, \mLin^0, \PLin^0)$ provided $\PLin^0$ is small enough.

\begin{prop} \label{prop:controll}
Assume Assumption \ref{assu:controlAffine} holds. For every $\mLin^0, \mLin^f \in \Rn$ and $\varepsilon >0$, there exists a constant $c>0$ such that,  for every non-nnegative symmetric matrix $\PLin^0$ with $\|\PLin^0\| \leq c$, there exists a control $u \in \mathbb{U}(\varepsilon, \mLin^0, \PLin^0)$ satisfying
\begin{equation} \label{eq:appConstr}
    \mLin_{u}(t_f) = \mLin^f .
\end{equation}
\end{prop}

\begin{rem}
Note that when the initial condition $x^0$ is deterministic, the initial covariance $\PLin^0$ is zero. The proposition then implies that, for every $\mLin^0, \mLin^f$ and $\varepsilon$, the corresponding set $\mathbb{U}(\varepsilon, \mLin^0, 0)$ contains controls joining $\mLin^0$ to $\mLin^f$. This case is important for applications.
\end{rem}

\begin{proof}
Fix $\mLin^0$, $\mLin^f$. Thanks to Assumption \ref{assu:controlAffine}, and in particular \eqref{eq:Lie}, there exist a constant $C=C (\mLin^0, \mLin^f )> 0$ and a control $\bar u \in L^2([0,1],\Rm)$ such that
$$
\mLin_{\bar u}(0) = \mLin^0 , \quad \mLin_{\bar u}(1) = \mLin^f , \quad \| \bar u(t) \| = C , \ t \in [0,1] .
$$
Therefore, for every $\eta > 0$ we may define the control $u_{\eta} \in L^2([0,t_f],\Rm)$ by
$$
u_{\eta}(t) \triangleq \frac{1}{\eta} \bar u\left( \frac{t}{\eta} \right) \mathds{1}_{[0,\eta]}(t) , \quad t \in [0,t_f] .
$$
It holds that
$$
\| u_{\eta} \|_{L^1} = \| \bar u \|_{L^1} = C ,
$$
and that
$$
\mLin_{u_{\eta}}(t) = \begin{cases}
    \displaystyle \mLin_{\bar u}\left( \frac{t}{\eta} \right) , \quad 0 \le t \le \eta \\
    \mLin^f , \quad \eta \le t \le t_f .
\end{cases}
$$
At this step, for every $u \in L^2([0,t_f],\Rm)$, the matrix $\PLin_u$ being solution of a Lyapunov differential equation may be explicitly computed as
$$
\PLin_u(t) = \Phi_u(t,0) \PLin^0 \Phi_u(t,0)^{\top}  +\int^t_0 \Phi_u(t,s) \G \G^{\top} \Phi_u(t,s)^{\top} \; \mathrm{d}s ,
$$
where the fundamental matrix $\Phi_u(t,s) \in \R^{n \times n}$ is the solution of
$$
\begin{cases}
    \displaystyle \frac{\mathrm{d}}{\mathrm{d}t} \Phi_u(t,s) = \sum^k_{i=1} u_i(t) D_x f_i(\mLin_u(t)) \Phi_u(t,s) \\
    \Phi_u(s,s) = I .
\end{cases}
$$
A straightforward application of Gr\"onwall's inequality yields
$$
\underset{0 \le s \le t \le t_f}{\sup} \| \Phi_u(t,s) \| \le e^{\bar C \| u \|_{L^1}} ,
$$
where the constant $\bar C > 0$ depends on $L$ uniquely by Assumption~\ref{assu:controlAffine}-\textit{2}. In what follows, we will implicitly overload the constant $\bar C$. For every $t \in [0,\eta]$ we may compute
\begin{align*}
\| \PLin_{u_{\eta}}(t) \|   & \le \| \Phi_{u_{\eta}}(t,0) \|^2 \|\PLin^0\| + \int^t_0 \| \G \|^2 \| \Phi_{u_{\eta}}(t,s) \|^2 \; \mathrm{d}s \\
& \le e^{\bar C \| \bar u_{\eta} \|_{L^1}} \left( \|\PLin^0\| + \bar C  t \right) = e^{\bar C C} \left( \|\PLin^0\| + \bar C  t \right),
\end{align*}
from which we finally obtain that (recall that $\varphi(r) = L r$)
\begin{align*}
  & \alpha \left( \int^{t_f}_0 \varphi(\| u_{\eta} \|)  \right) \times   \int^{t_f}_0 \varphi(\| u_{\eta} \|) \| \PLin_{u_{\eta}} \|  \\
&= \alpha(LC)  \int^{\eta}_0 \frac{L}{\eta} \left\| \bar u\left( \frac{t}{\eta} \right) \right\|  \; \| \PLin_{u_{\eta}}(t) \|  \; \mathrm{d}t \\ 
    & \le \mathrm{Const} \times ( \eta +  \|\PLin^0\|).
\end{align*}
The conclusion follows from the arbitrariness of $\eta > 0$.
\end{proof}

Thus we should solve Problem \ref{pb:stat_lin} by optimizing over controls in $\mathbb{U}(\varepsilon, \mLin^0, \PLin^0)$ if we desire to guarantee that the trajectories of this problem consist of $\varepsilon$-approximations of the ones solutions to Problem~\ref{pb:sto}. Proposition \ref{prop:controll} ensures that under the constraint $u \in \mathbb{U}(\varepsilon, \mLin^0, \PLin^0)$ Problem \ref{pb:stat_lin} remains feasible. In practice it is preferable either to treat constraint~\eqref{eq:approx_constraint} via penalization over controls, or simply to check at a post-processing level that the left-hand side of \eqref{eq:approx_constraint} stays small. We use the latter approach in the next section. 



\section{Application: motion planning of the powered descent of a space vehicle}
\label{se:launcher}
In this section, we present an application of our method to the problem of the powered descent of a space vehicle. In particular, we show the interest of our robust motion planning approach to optimize landing trajectories under uncertainties. Landing trajectories are often performed by vertical landing, whose last stage, the powered descent, is particularly challenging and requires high accuracy \cite{Blackmore2016}. To achieve this, the usual strategy is to optimize a reference trajectory in the first place that will be tracked during the powered descent thanks to a feedback control \cite{Yu2017}. Up to now, the trajectory optimization has been done in general in a deterministic optimal control framework, without considering uncertainties. For instance, \cite{Meditch1964} has first introduced the landing trajectory generation problem as a minimal-fuel optimal control problem, and since then, a wide literature has investigated landing problems. Existing results state that the optimal control has generally a Max-Min-Max form \cite{Leparoux2022c, Lu2018, Gazzola2021}, meaning that its norm switches at most twice between its inferior and superior bounds. This control is particularly sensitive to uncertainties, hence the interest in developing robust methods for which are subject to high performance criteria. 

\subsection{Framework}
 For the sake of simplicity, we consider a two dimensional formulation only, we refer the reader to \cite{Leparoux2022c} for a discussion on the complete three dimensional model. In this setting, the state $x=(r, v, \m) \in \R^5$ is composed by the position $r=(y,z) \in \R^2$, the velocity $v=(v_y,v_z) \in \R^2$, and the mass $\mu \in \R$, while the control $u=(u_y,u_z)$ represents the two dimensional thrust.

Ignoring disturbances and uncertainties, the dynamics of the system, referred later as the \emph{unperturbed dynamics}, is
\begin{equation}\label{eq:dynVL}
\dot x = f(x, u) =
\begin{pmatrix}
v \\
\frac{T}{\m} u - 
\begin{pmatrix}
  0 \\ g_0
\end{pmatrix}
 \\
-q \Vert u \Vert
\end{pmatrix}
\end{equation}
where the maximal thrust $T$, the mass flow rate $q$ of the engine, and the gravitational acceleration $g_0$ are  positive constants.
Moreover, the thrust is subject to the physical limits of the actuators, which can be modeled by upper and lower bounds on the control norm as follows
\begin{equation}\label{eq:ctenorm} 
0 \leq u_{min} \leq \| u(t) \| \leq u_{max}. 
\end{equation}
The goal of the powered descent is to reach the target $\Sf=\{r=0, v=0\}$ from the initial state $x(0)$. To save the largest amount of fuel in the vehicle tanks, the motion planning problem takes the form of Problem~\ref{prob1} with the following cost function:
$$
C(u) = -\m(t_f).
$$
In addition, the system is subject to at least two types of uncertainties. First, several effects such as aerodynamic forces are usually neglected from the dynamics model for convenience. Therefore, we propose to model these neglected quantities through an additive noise on the acceleration (i.e., the derivative of $v$). As a result, we write the dynamics as
\begin{equation} \label{eq:trueSDE1}
    \mathrm{d}x_t = f(x_t,u(t)) \, \mathrm{d}t + \G \, \mathrm{d}W_t,
\end{equation}
where $W_t$ is a one-dimensional Brownian motion and $\G = ( 0, 0, \sigma_y, \sigma_z, 0)$ is a constant vector, $\sigma_y, \sigma_z$ being nonnegative constants. Second, the initial state $x(0)$ is known only up to measurement errors. We model these errors by considering $x(0)$ as a random variable $x^0$ with normal density $\mathcal{N}(\mTr^0,\PTr^0)$. Finally, we formulate the robust motion planning of the powered descent as a problem in the form of Problem~\ref{pb:sto}. 

\subsection{Direct statistical linearization}
\label{sse:framework}

Following the approach described in Section \ref{se:robustMP}, we solve a problem in the form of Problem~\ref{pb:stat_lin} rather than one in the form of Problem~\ref{pb:sto}. In particular, we look for a solution of the robust powered descent problem as a control law $u$ which minimizes the following problem (where we write $m$ as $m=(m_r,m_v,m_\mu) \in \R^2 \times \R^2 \times \R$).
%
%
%
%

\begin{prob}\label{pb:mpSV}
\[
\min \limits_{u, t_f} \quad J_{lin}(u) = -\mLin_{\mu}(t_f) + \mathrm{tr} (Q_f\PLin(t_f)) + \int_0^{t_f} \mathrm{tr} (Q\PLin(t)) dt
\]
under the constraints
\[
\begin{cases}
	\begin{array}{l}
			(\mLin , \PLin)(\cdot) \ \text{solution of the statistical linearization \eqref{eq:syslin}}, \\
			u(t)  \in \mathcal{U}=\left \lbrace u \in \R^2 \ | \  u_{min} \leq \| u \| \leq u_{max} \right \rbrace \quad \hbox{for a.e.} \ t \in [0, t_f], \\
			(\mLin,\PLin)(0) = (\mTr^0,\PTr^0), \\
			\mLin(t_f) \in \Sf.
	\end{array}
\end{cases}
\]
\end{prob}

Unfortunately, this formulation can not be leveraged because the corresponding statistical linearization is not controllable, in the sense that it does not enjoy accessibility property. In particular, there is no reason for the solution of Problem \ref{pb:mpSV} to perform the powered descent with a small final covariance.

Let us better elucidate this ``non controllability'' property. We introduce first some notations in a general setting. Given a dynamical system of the form~\eqref{eq:dyn} on $\R^n$, denote by $\f = \{ f(x,u) \ : \ u \in \mathcal{U} \}$ the associated family of vector fields, and consider the following linear space of vector fields,
 \begin{equation}\nonumber
\I(\f) = \Vect {f_{1} - f_{2}, f_3 : f_1, f_2 \in \f, f_3 \in \mathcal{D}(\f) },
\end{equation}
where $\mathcal{D}(\f) = \lbrace \left[ f_{1}, ..., [f_{k-1},  f_{k}] \right]: k\geq 2, f_i \in \f \rbrace$. Recall that a dynamical system is said to be \textit{accessible} from $x \in \R^n$, for some $x \in \R^n$, if the set of reachable points from $x$ in fixed time has a non empty interior. We have proven in \cite{bonalli2022} that, 
if for every $\hx \in \R^n$ there holds
 \begin{equation}
 \label{eq:suffcondition2}
    \dim \left\{ \vecdef{f(m)}{Df(m) + Df(m)^{\top}} \ : \ f \in \I(\f) \right\} = n + \frac{n(n+1)}{2},
 \end{equation}
then the statistical linearization~\eqref{eq:syslin} defined by $f(x,u)$ and any dispersion $g$ is accessible from any point in an open and dense subset of $\R^n \times \Sym$.

Note that in the present application this condition is never satisfied.

\begin{lem}
For the dynamic system \eqref{eq:dynVL}, no point $m \in \R^5$ satisfies condition \eqref{eq:suffcondition2}.
\end{lem}
\begin{proof}
As the proof consists of a simple computation, below we only detail the main steps. For the dynamic system \eqref{eq:dynVL}, any vector field  $f_u \in \f$ has the form
\begin{equation} f_u(x) = \begin{pmatrix} v_y \\ v_z \\ f_{u,3}(\m)  \\ f_{u,4}(\m) \\ f_{u,5} \end{pmatrix}\end{equation}
where $f_{u,3}(\m) = \frac{T}{\m} u_y$, $f_{u,4}(\m) = \frac{T}{\m}u_z -g_0$ and $f_{u,5} = -q \|u\|$. Thus, for any $u_1, u_2 \in \U$, the  first two components of $f_{u_1}(x) - f_{u_2}(x)$ are zero, as well as the last component of $[f_{u_1}, f_{u_2}](x)$, and their other components are functions of $\mu$ only.
By induction, we deduce that all elements $f \in \I(\f)$ are functions of $\m$, so that 
\begin{equation}
\text{dim}\left(\Vect{\vecdef{f(x)}{Df(x) + Df(x)^{\top}}  \ : \  f \in \I(\f)}\right) \leq 9< n + \frac{n(n+1)}{2},
\end{equation}
and the conclusion follows.
\end{proof}
\begin{rem}
However, in a similar fashion we can prove the unperturbed system is accessible from any point, by showing that the dimension of $\text{Vect} \lbrace f(x) \ :$ $ f \in \I(\f)\rbrace$ is equal to $n$. Indeed, by considering for instance the controls $u_1 = (0,0)$, $u_2 = (0,1)$, $u_3 = (1,0)$ and $u_4 = (\frac{1}{\sqrt{2}}, \frac{1}{\sqrt{2}})$, we obtain that the vectors $f_{u_1} (x) - f_{u_2}(x)$, $f_{u_1}(x) - f_{u_3}(x)$, $f_{u_1}(x) - f_{u_4}(x)$, $ \left[f_{u_1},f_{u_2} \right](x)$, $\left[f_{u_1}, f_{u_3} \right](x)$ are linearly independent, which leads to the conclusion.
\end{rem}
The sufficient accessibility condition is not satisfied. This does not prove that the statistically linearized system \eqref{eq:syslin} is not accessible, since the condition tested is only a sufficient one. However, the lack of accessibility is expected because here the vector fields in $\f$ are analytic vector fields, and in this case the Lie-based accessibility sufficient conditions are merely necessary (see~\cite{Jurdjevic1996}). In a more informal way, notice that the dynamics \eqref{eq:dynVL} is almost linear and the statistical linearizations of linear systems are never accessible. Finally, our numerical solutions of the robust motion planning problem subject to dynamics \eqref{eq:syslin} (such as those presented on Figure \ref{fi:plotOL}) do not succeed in reducing the covariance, which tends to confirm this non accessibility property. 

\subsection{Statistical linearization through partial feedback}

\subsubsection{Adding feedback into the model} \label{sse:accFB}

We have seen in the previous subsection that our method fails to provide an open-loop control that is satisfactory in terms of robustness. Nevertheless, some measurements, such as for position $r$ and velocity $v$, are available during the powered descent, and we can include those into a feedback control. Note that, on the one hand, measurements of mass are generally not available and, on the other hand, this quantity is in general not observable, which in turn hinders the design of feedback controls that are also functions of the mass. 

Therefore, we replace the control variable $u$ in \eqref{eq:dynVL} by a function $u_{FB}$ of $\overline x = (r, v)$ which depends on some appropriate parameters $\nu$, and in turn we consider $\nu$ as the new control variable. When including this partial feedback controls, the launch vehicle unperturbed dynamics writes as
\begin{equation}  \label{eq:dynFB}
f_{FB} (x, \nu) = f(x, u_{FB}(\overline x, \nu)).
\end{equation}
 Keeping the same dispersion term $g$ as in~\eqref{eq:trueSDE1}, we obtain a new stochastic model for the dynamics,
\begin{equation} \label{eq:trueSDE2}
    \mathrm{d}x_t = f_{FB}(x_t,\nu(t)) \, \mathrm{d}t + \G \, \mathrm{d}W_t.
\end{equation}
Returning to the physics of the problem, let us remark that the control $u$ is provided by two actuators, one for the norm of $u$, the other one for its direction, so that 
\begin{equation} \label{eq:urt}
u = u_\rho\begin{pmatrix} \cos(u_\theta) \\ \sin(u_\theta)\end{pmatrix},
\end{equation}
where $u_\rho = \| u\|$  and $u_\theta \in [-\pi, \pi)$. 
Restricting ourselves to linear functions of $\overline{x}$, we write the feedback control as
\begin{equation}\label{eq:uFB}
    u_{FB}(x,\nu) = (\rho + K_n \overline{x})\begin{pmatrix} \cos(\theta + K_d \overline{x}) \\ \sin(\theta + K_d \overline{x})\end{pmatrix}
\end{equation}
where  $\rho, \theta \in \R$, $K_n,K_d \in \R^4$, so that $\nu = (\rho, \theta, K_n, K_d)$.

Note that the statistical linearization of the system defined by \eqref{eq:trueSDE2} writes as
\begin{equation} \label{eq:syslinFB}
\left\{ \begin{array}{ll}
\dot{\hx} &= f_{FB}(\hx,\nu),\\[2mm]
\dot{P} &= D_xf_{FB}(\hx,\nu) P + P D_xf_{FB}(\hx,\nu)^{\top} + \G \G^{\top},
\end{array} \right. 
\end{equation}
and we checked with the help of a formal calculation software that, unlike in the case of the statistical linearization of~\eqref{eq:trueSDE1}, this new system satisfies the sufficient condition for accessibility, i.e.\ that, denoting by $\f_{FB}$ the family of all vector fields $f_{FB}(\cdot, \nu)$, there holds
\begin{equation}
\dim \left(\Vect{\vecdef{f(x)}{Df(x) + Df(x)^{\top}} \ : \  f \in \I(\f_{FB})}\right) = 20 = n + \frac{n(n+1)}{2}, 
\end{equation}
for any $x$ in $\R^n$, with $n = 5$.

\subsubsection{Adding actuator limits to the model}\label{sse:satmod}

In the dynamics~\eqref{eq:trueSDE2}, the control of the actuators $u_{FB}(x_t, \nu(t))$ becomes a stochastic process, therefore imposing $u_{min} \leq \| u_{FB} \| \leq u_{max}$ is not possible. In turn, the physical limit of the actuators have to be taken into account differently in the model, and we propose two independent ways to handle them.

\paragraph{First approach: saturation modelling}
The first approach consists in encoding the physical limits of the actuators directly in the dynamic model. This can be achieved by using a saturation function in the model, which can be given as follows: for real numbers $a<b$, consider a function $\mathrm{sat}_{a}^{b}$ such that $\mathrm{sat}_{a}^{b}(s) \in [a,b]$ for every $s \in \R$.
Then, using the notation of the control introduced by \eqref{eq:urt}, we replace the original unperturbed dynamics $f(x,u)$ by
\begin{equation}\label{eq:dynsat}
f^{\mathrm{sat}}(x,u) = f\left(x, \mathrm{sat}_{u_{min}}^{u_{max}}(u_\rho)\begin{pmatrix} \cos( u_\theta) \\ \sin( u_\theta) \end{pmatrix}\right),
\end{equation}
and then replace $f_{FB} (x, \nu) = f(x, u_{FB}(x, \nu))$ in~\eqref{eq:trueSDE2} by 
$$
f^{\mathrm{sat}}_{FB} (x, \nu) =f^{\mathrm{sat}}(x, u_{FB}(x, \nu)).
$$

Assuming that $\mathrm{sat}_{a}^{b}$ is a smooth function, we can compute the statistical linearization of the saturated feedback dynamics as 
\begin{equation} \label{eq:syslinsat}
\begin{cases}
\dot{\hx} &= f^{\mathrm{sat}}_{FB} (\hx, \nu),\\[2mm]
\dot{P} &= D_xf^{\mathrm{sat}}_{FB} (\hx, \nu) P + P D_xf^{\mathrm{sat}}_{FB} (\hx, \nu)^{\top} + \G(m) \G(m)^{\top}.
\end{cases}
\end{equation}

Following our previous discussion, we add the aforementioned control constraints to Problem \ref{pb:stat_lin}, obtaining the following robust motion planning problem.

\begin{prob}[Saturated closed-loop robust motion planning]
\label{pb:VLscoc}
\[
\min \limits_{\nu, t_f} \quad -m_{\mu}(t_f) + \mathrm{tr} (Q_fP(t_f)) + \int_0^{t_f} \mathrm{tr} (Q P(t)) dt
\]
under the constraints
\[
\begin{cases}
	\begin{array}{l}
			(\hx,P)(\cdot) \ \text{follows \eqref{eq:syslinsat}}, \\
			(\hx,P)(0) = (\hx^0,P^0), \\
			\hx(t_f) \in \Sf. \\
	\end{array}
\end{cases}
\]
\end{prob}

\begin{rem}
Note that the smoothness of $\mathrm{sat}_{a}^{b}$ is required in order to use statistical linearization. Different possibilities exist for the choice of a smooth saturation function, e.g., the hyperbolic tangent function or the exact (non smooth) saturation function defined by $\overline{\mathrm{sat}}_{a}^{b}(s)=s$ if $s \in [a,b]$, $=a$ if $s\leq a$, and $=b$ if $s\geq b$ (we use this latter in our simulation). Note that this function can be written as
\begin{equation}\label{eq:sat2}
   \overline{\mathrm{sat}}_{a}^{b}(s) = \frac{b + a}{2} + \frac{|s - a| - |s - b  |}{2}.
\end{equation}
We then choose a small parameter $\epsilon$ and define $\mathrm{sat}_{a}^{b}$ by replacing the absolute values $|s|$ by $\sqrt{s^2 + \epsilon^2}$ in the expression above. The smaller $\epsilon$ is, the closer the approximation to the exact saturation function is (see Figure \ref{fi:plotAprox}). 
\begin{figure}[] 
\centering
\begin{tikzpicture}[scale=1, baseline]
\pgfplotssetlayers
\pgfplotsset{every axis/.append style={
after end axis/.code={
};
}}
\begin{axis}[%
 	mark size = 1pt,
    width= 7 cm,
    height = 5 cm,
	grid = major,
	xlabel = $u$,
	xticklabels={,,},
	yticklabels={,,},
	xmin =-0.5,
	xmax = 1.5,
	xtick       = {0.2,0.8},
    xticklabels = {$u_{min}$,$u_{max}$},
    ytick       = {0.2,0.8},
    yticklabels = {$u_{min}$,$u_{max}$},
	ylabel = $\text{sat}_{u_{min}}^{u_{max}}(u)$,	
	title = {Saturation function for $u_{min} = 0.2$ and $u_{max} = 0.8$},
	legend style={ 
	       at={(0.5,-0.7)},
                    anchor=south,
                    legend columns= 2,
                        }
]

\def \umax{0.8}
\def \umin{0.2}

\def \e{0}
\addlegendentry{ Exact}
\draw[ thick, black] (-0.5, 0.2) -- (\umin, \umin);
\addplot[thick, black] coordinates {(\umin, \umin) (\umax, \umax)};
\draw[ thick, black] (\umax, \umax) -- (2, \umax);
\def \e{0.2}
\addlegendentry{ $\epsilon = 0.2$}
\addplot [ thick, red, samples =200]   { (\umax + \umin)/2 + (sqrt((\x- \umin)^2 + (\e)^2) - sqrt((\umax - \x)^2 + (\e)^2 ))/2}  ;

\def \e{0.1}
\addlegendentry{  $\epsilon = 0.1$}
\addplot [ thick, blue, samples =200]  { (\umax + \umin)/2 + (sqrt((\x- \umin)^2 + (\e)^2) - sqrt((\umax - \x)^2 + (\e)^2 ))/2} ;
\def \e{0.01}
\addlegendentry{$\epsilon = 0.01$}
\addplot [ thick, black!30!green, samples =200]  { (\umax + \umin)/2 + (sqrt((\x- \umin)^2 + (\e)^2) - sqrt((\umax - \x)^2 + (\e)^2 ))/2}  ;

\end{axis}
\end{tikzpicture}
\captionsetup{justification=centering} 
\captionof{figure}{The saturation function: exact value and smooth approximations}
\label{fi:plotAprox}
\end{figure}
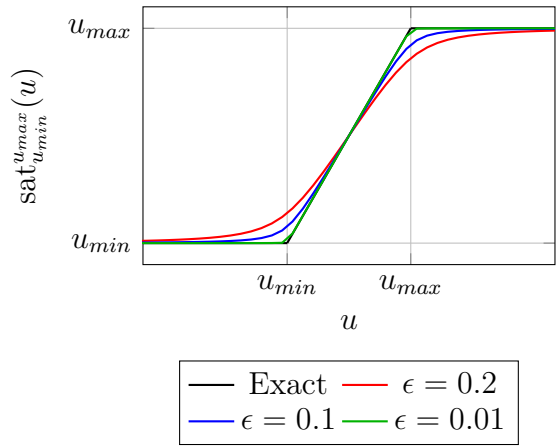
\end{rem}

\paragraph{Second approach: chance constraints}
The second approach consists in requiring that the inequality constraint \eqref{eq:ctenorm} be satisfied with a probability greater than a certain given threshold. Let us explain why such chance constraint is well-suited for the statistical linearization approach. Indeed, consider a chance constraint such as
\begin{equation} \label{eq:probaCte}
    \Pr[a^{\top}x_t \leq c] \geq p,
\end{equation}
where $a \in \R^5$, $c \in \R$, and $p \in (0,1)$ is the chosen probability threshold. In this formula the probability $\Pr$ is supposed to be computed with respect to the probability distribution of the process $x_t$ solution of~\eqref{eq:trueSDE2}. Now, the main principle underneath the method of statistical linearization consists of approximating the distribution of $x_t$ by rather a normal distribution whose mean and covariance are the state variables $\mLin(t)$ and $\PLin(t)$ of the statistical linearization~\eqref{eq:syslinFB} (see \cite{Berret20202} for explanations on this approximation). For such a distribution the chance constraint~\eqref{eq:probaCte} writes as
\begin{equation}
    a^{\top}\mLin+ \Psi^{-1}(p)\sqrt{a^{\top}\PLin a} \leq c,
\end{equation}
where $\Psi^{-1}$ is the inverse cumulative distribution function of the normal distribution. Thus, the chance constraint on the stochastic process is transformed into a state constraint in the setting of the statistical linearization.

Following this approach, we reformulate the norm constraint on $u_{FB}$ as the following chance constraint on the solutions of \eqref{eq:trueSDE2}: for every $t \in[0, t_f]$, there holds
\begin{equation}\label{eq:ctenormPb} 
\Pr(  \|u_{FB}(x_t,\nu)\| \in [u_{min}, u_{max}] )= \Pr(  (\rho + K_n \overline{x}_t) \in [u_{min}, u_{max}] ) \geq p,
\end{equation}
where $p$ is the chosen threshold. By replacing these constraints by state constraints on the statistical linearization, we finally formulate the following robust motion planning problem.

%
%

\begin{prob}[Chance constrained robust motion planning]\label{pb:mpProba}
\[
\min \limits_{\nu, t_f} \quad -\mLin_{\mu}(t_f) + \mathrm{tr} (Q_f\PLin(t_f)) + \int_0^{t_f} \mathrm{tr} (Q \PLin(t)) dt
\]
under the following constraints:
\begin{itemize}
  \item $(\hx,P)(\cdot)$ is solution of~\eqref{eq:syslinFB} with $(\hx,P)(0) = (\hx^0,P^0)$;
  \item $\forall t \in[0, t_f]$ there holds
$$
\left\{
\begin{array}{l}
\displaystyle u_{max} -(\rho(t) + K_n(t) \overline{\hx}(t))  \geq \Psi^{-1}(p) \sqrt{K_n(t) \overline{P}(t) K_n(t)^{\top}}, \\[2mm]
\displaystyle	        (\rho(t) + K_n(t) \overline{\hx}(t) ) - u_{min} \geq \Psi^{-1}(p) \sqrt{K_n(t) \overline{P}(t) K_n(t)^{\top}},
\end{array} \right.
$$
where $\overline \mLin$ denotes the first four components of $\mLin$ and  $\overline \PLin$  the matrix formed by the first four lines and columns of $\PLin$;
  \item $\hx(t_f) \in \Sf$.
\end{itemize}
\end{prob}

Note that this formulation may be numerically expensive to handle because additional mixed constraints on both mean and covariance are considered.

\begin{rem}
The solution $\nu(\cdot)$ of the above problem guarantees that the constraint $u_{min} \leq \| u_{FB} \| \leq u_{max}$ is satisfied only in probability. Thus when using this control law $\nu(\cdot)$ for simulating the stochastic model~\eqref{eq:trueSDE2}, one has to model the system including saturation (such as the function $f^{\mathrm{sat}}_{FB}$ defined in the first approach) to force the physical limits of the actuator. 
\end{rem}

\subsection{Numerical results}\label{sse:simus}

This section presents some numerical results to illustrate solutions of Problems \ref{pb:mpSV}--\ref{pb:mpProba}. Calculations are based on a direct method which makes use of a time discretization of the considered optimal control problems, using a grid of $150$ nodes and the CasADi toolbox \cite{Andersson2019} (combined with the IPOPT solver). The computed open-loop optimal control is then used as an input to simulate the stochastic dynamic model with random uncertainties. When performing the simulations, actuator limits are forced by applying a saturation on the input control following the exact expression \eqref{eq:sat2}.

\begin{figure}

\begin{tabular}{| c c |}
 \hline
 Property & Value \\ 
 \hline
 $T$ & $1e6N$ \\  
 $u_{min}$ & 0.2   \\
 $u_{max}$ & 0.8       \\
 $q$ &$300 kgs^{-1}$\\
 $g_0$ & $9.81ms^{-2}$\\
 $r^0$ & $(1000, 4000)m $ \\  
 $v^0$ & $(-75, -200) ms^{-1}$ \\
 $\mu^0$ & $40000 kg$    \\
 \hline
\end{tabular} \hspace{0.5cm}
\begin{tabular}{| c c |}
 \hline
 Property & Value  \\ 
 \hline
  $g$ & $\text{diag}(0, 0, 100, 10, 0)N $ \\
 $P^0$ & $\text{diag}(100m^2, 100m^2, $ \\
 & $ 1m^2s^{-2}, 1m^2s^{-2}, 1600kg^2)$ \\
 $Q$ & $\text{diag}(10, 50, 1, 10, 0).10^3$ \\
 $Q_f$ & $\text{diag}(14, 20, 0.2, 4, 0).10^3$ \\
 \hline
\end{tabular}

\captionof{table}{Simulation settings}
\label{ta:params}
\end{figure}

The parameters in the dynamics \eqref{eq:dynVL} and the simulation settings are specified in Table \ref{ta:params}. The initial mean state is given by $(r^0 , v^0, m^0)$. We assume that measurements of the initial position and velocity are quite accurate, while the initial mass is imprecisely known, hence the setting of $P^0$. Along the trajectory, the dynamics is subject to perturbations due to aerodynamic effects, modeled by white noises of standard deviation depending on the mass $\G = \frac{\sigma}{\m}$. Moreover, the nonnegative symmetric matrices $Q$ and $Q_f$ are fixed to the same value for the three considered problems. We do not penalize the mass covariance, since position and velocity are the quantities of interest to attain precision landing and ensure security.

Figure \ref{fi:plotOL} shows Monte-Carlo samples of trajectories solutions to dynamics \eqref{eq:trueSDE1} which are generated by the control solution of Problem \ref{pb:mpSV}. The right plot shows the norm $||u||$ of the optimal control, which takes a Min-Max form. The left plot shows simulated trajectories. We clearly observe that the covariance in position increases with time although it is strongly penalized in the cost, which tends to confirm that there is a lack of accessibility using this formulation. The final state standard deviation is $\overline{P}(t_f) = (29.7m, 58.6m, 1.3ms^{-1}, 5.2ms^{-1})$, which is indeed greater than the initial covariance. Note that the final time $t_f$ is $26.4s$.

\begin{figure}[h!]
    \centering
    \includegraphics[width = 12cm]{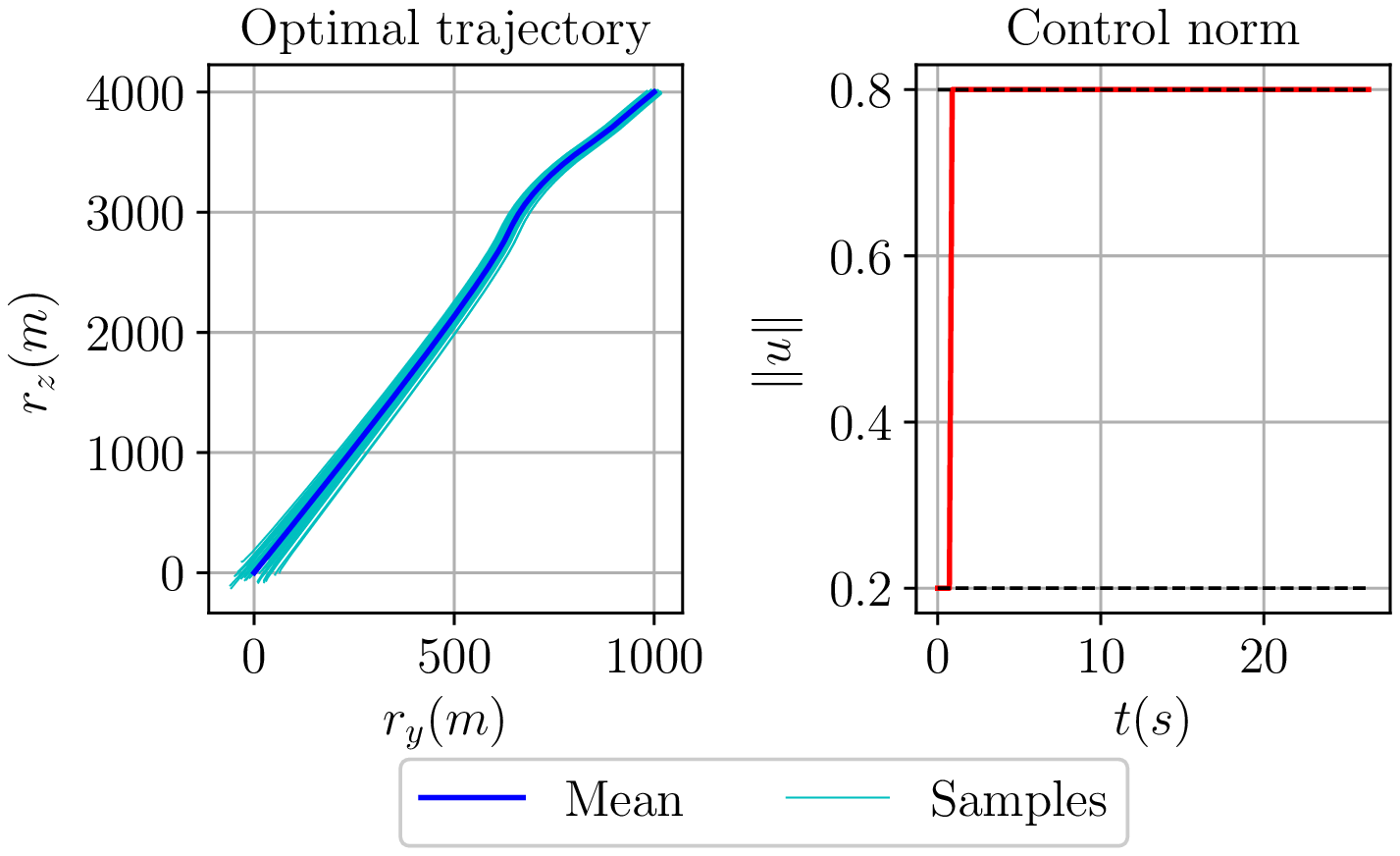}
\captionof{figure}{Simulations of random trajectories minimizing the covariance, with a deterministic control}
\label{fi:plotOL}
\end{figure}


Figure \ref{fi:plotProba} shows trajectories solutions of dynamics \eqref{eq:trueSDE2}, which are generated by the control solution of Problem \ref{pb:mpProba} with $p = 0.99$. The feedback gains are penalized within the cost to regularize the problem:
\begin{equation} \label{eq:costCL}
\min \limits_{\nu, t_f} \quad -\hx_\m(t_f) + \text{tr}(Q_fP(t_f)) + \int_0^{t_f} \text{tr}(Q\PLin(t)) + 2\|K_n(t)\|^2 + \|K_d(t) \|^2  dt .
\end{equation}
The shape of the control norm $\| u \|$ looks like a Max-Min-Max control to which margins would have been added with respect to the real actuator limits, $u_{min}$ and $u_{max}$. 
We observe that the dispersion is controlled all along the trajectory, as required. The final state standard deviation is $\overline{P}(t_f) = (5.2m, 5.8m, 0.5m/s,$ $ 0.5m/s)$, which is smaller than the initial covariance. The final time is $34.4s$, which is much higher than the final time of the previous setting. This can be explained as follows. Without the accessibility property, reducing the final time represents the only way to minimize the final covariance. On the contrary, thanks to 
partial feedback the control may be kept closer to the fuel-optimal solution while providing robustness.

\begin{figure}[h!]
    \centering
    \includegraphics[width = 12cm]{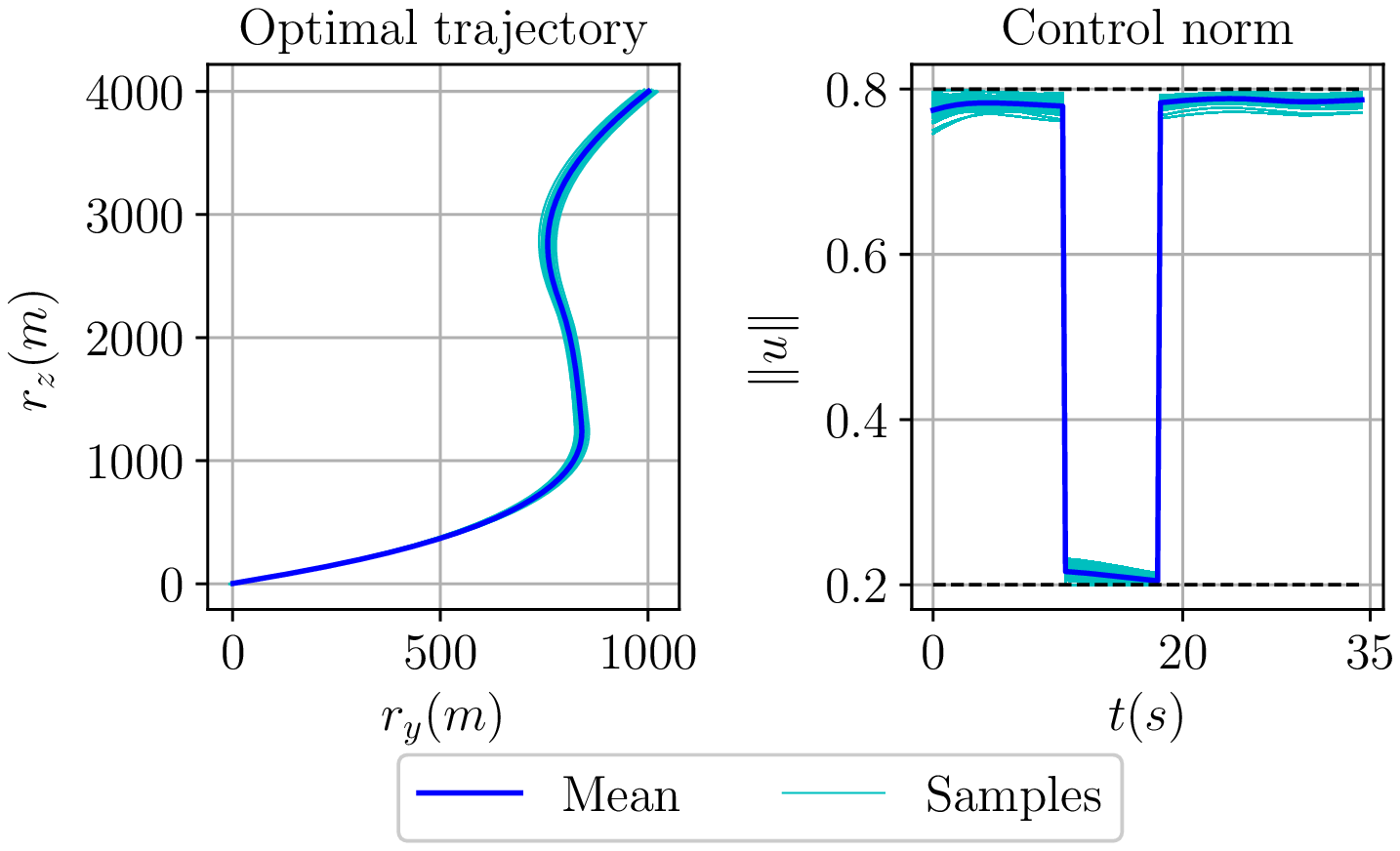}
\captionof{figure}{Simulations of random trajectories minimizing the covariance, with feedback dynamics and chance constraints}
\label{fi:plotProba}
\end{figure}

Figure \ref{fi:plotSat} shows trajectories solutions of dynamics \eqref{eq:trueSDE2}, which are generated by the control solution of Problem \ref{pb:VLscoc}, with partial feedback and actuator limits treated as a saturation in the model. The smoothing parameter $\epsilon$ is set to $0.02$. The cost is still expressed by \eqref{eq:costCL}.
The final state standard deviation is $\overline{P}(t_f)=(4.7m, 4.7m, 0.4m/s,$ $ 0.7m/s),$ and its norm is even smaller than the norm of the solution which stems from the chance constraint formulation. The final time ($34.7s$) is very close to the one obtained with chance constraints. Importantly, we note that the controls obtained with the two feedback methods are very similar in shape and duration, except that the control obtained under smooth saturation approximation is smoother, yielding larger margins on the norm of the control $||u||$.


\begin{figure}[h!]
    \centering
    \includegraphics[width = 12cm]{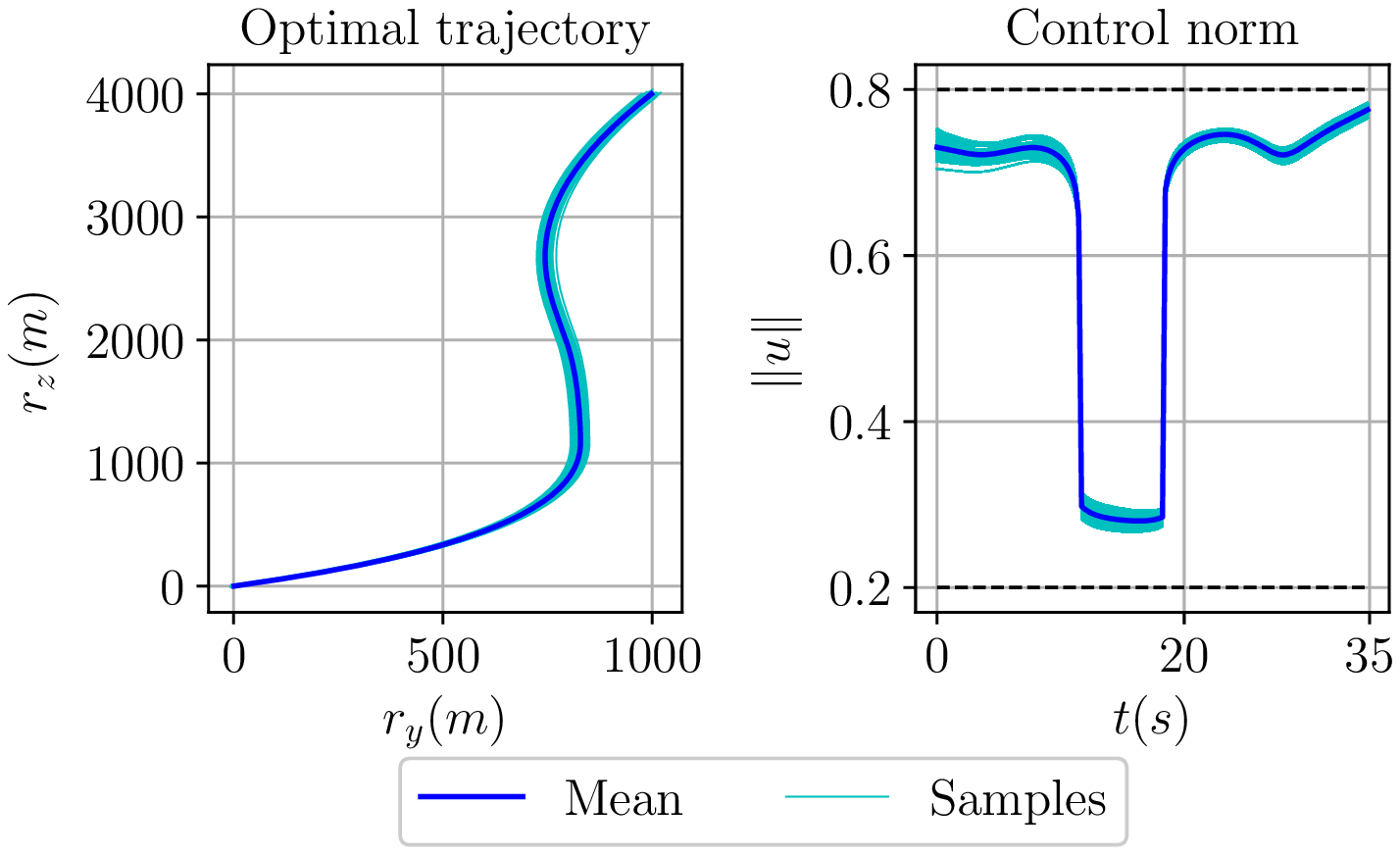}
\captionof{figure}{Simulations of random trajectories minimizing the covariance, with smoothly saturated feedback dynamics}
\label{fi:plotSat}
\end{figure}

Finally, the fourth set of plots (Figure \ref{fi:plotErr}) shows the relative error stemming from statistical linearization. Those are obtained when simulating the saturated feedback linearized dynamics \eqref{eq:syslinsat} using the control solution of Problem \ref{pb:VLscoc}. Approximation of the true mean and covariance are obtained via Monte-Carlo estimates with $1000$ sample trajectories solutions to dynamics \eqref{eq:trueSDE2}. Thus, since the errors on both mean and covariance remain bounded, we deduce that estimates of the statistical linearization are consistent. 

\begin{figure}[h!]
    \centering
    \includegraphics[width = 12cm]{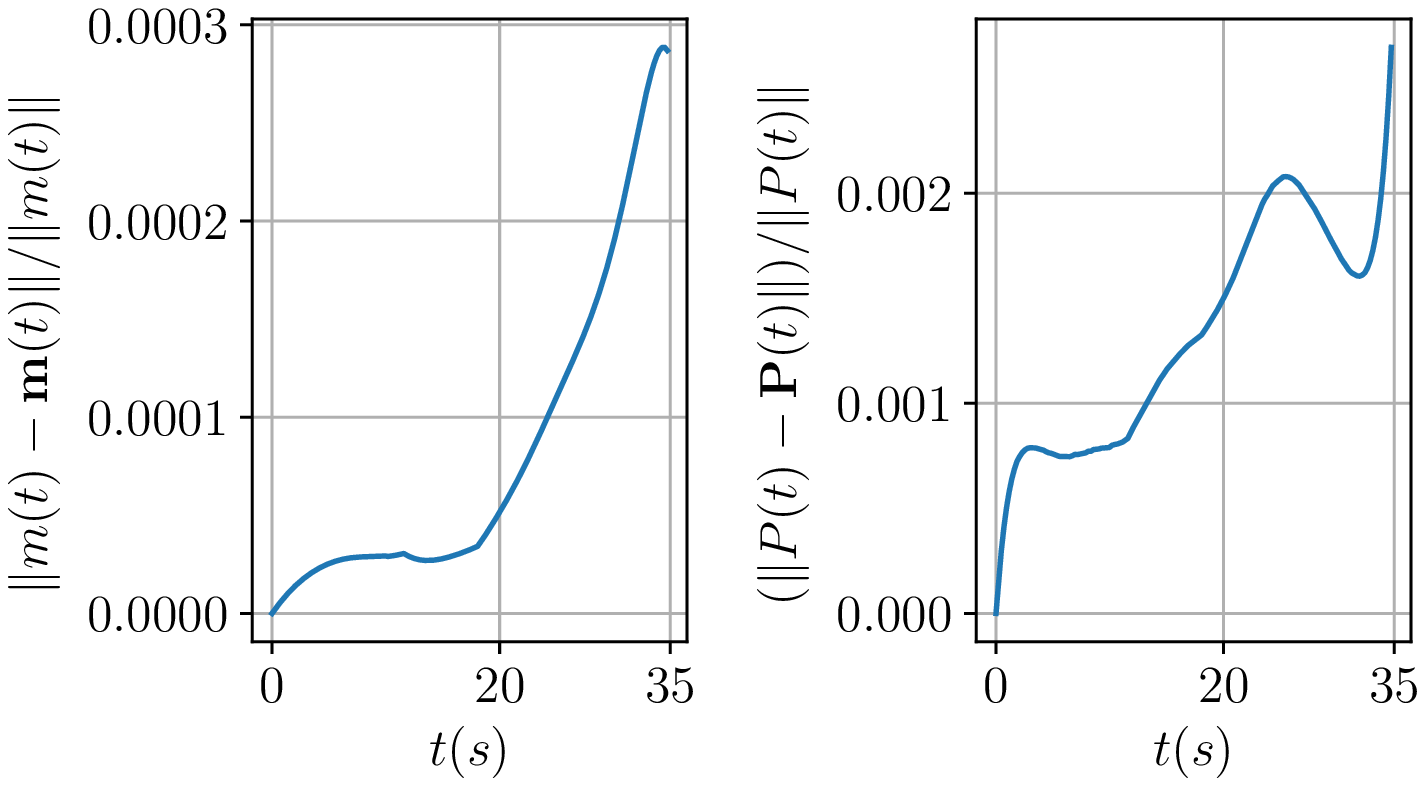}
\captionof{figure}{Relative errors on statistical linearization estimates of the mean and the covariance}
\label{fi:plotErr}
\end{figure}







\section{Conclusion and perspectives}
\label{se:conclusion}
We presented a method for stochastic robust motion planning which leverages statistical linearization to approximate the original formulation with a deterministic optimal control problem on the mean and the covariance of the original state variables. We justify our work through appropriate theoretical bounds for the approximation error due to statistical linearization, and through numerical experiments on the powered descent of a space vehicle.

We suggest three main future research directions. First, we will investigate extensions of our theoretical bounds for the approximation error due to statistical linearization to more general settings, e.g., to stochastic systems whose diffusion explicitly depends on the state variables. Second, we will consider extending our controllability results in Section \ref{sec:contr} to settings which go beyond control-linear systems; one possible direct application of this latter result would encompass controllability of stochastic differential equations. Third, we will focus on applying our approach to other applications, such as stochastic robust motion planning of autonomous vehicles and space robots. 

\bibliographystyle{abbrv}
\bibliography{main.bib}
\end{document}